\pdfoutput=1
 
\documentclass[12pt]{amsart}

\usepackage[mathscr]{eucal}
\usepackage{amscd}
\usepackage{amsfonts}
\usepackage{amsmath, amsthm, amssymb}
\usepackage{latexsym}
\usepackage{needspace}
\usepackage{hyperref}

\title{A Local Relative Trace Formula for $F^{\times} \backslash SL_2(F)$}
\author{Jonathan Sparling}

\newtheorem{thm}{Theorem}[section]
\newtheorem{lemma}[thm]{Lemma}

\newtheorem{prop}[thm]{Proposition}

\theoremstyle{definition}
\newtheorem{definition}[thm]{Definition}

\theoremstyle{remark}

\begin{document}

\begin{abstract}
In this note, we derive explicitly the local relative trace formula for the symmetric space $F^{\times} \backslash SL_2(F)$ at the level of Lie algebras, where $F$ is a $p$-adic field of residue characteristic greater than two and characteristic zero. This is perhaps one of the simplest non-trivial analogs of the trace formula.
\end{abstract}
\maketitle

\section{Introduction}

\label{1}

\subsection{Preamble}

The relative trace formula is a tool for studying harmonic analysis on reductive symmetric spaces in the same way that Arthur's trace formula is a tool for studying harmonic analysis on reductive groups.  It was introduced by Jacquet and Lai \cite{JL85}, building on ideas closely related to the Kuznetsov trace formula; see \cite{CPS90} for the latter and \cite{JLR93} for subsequent developments for symmetric spaces.  Ever since, it has generated a large amount of work, some of which is surveyed in \cite{Jac05}, although at the time this article was first prepared the local theory was still far from completely developed.  The goal of this article is to provide an explicit motivating example for this extension of the local trace formula to symmetric spaces at the level of Lie algebras.

More precisely, the local trace formula of Arthur (see \cite{Art91}) has already been used by Waldspurger in \cite{Wal95} to derive a local trace formula for a Lie algebra $\mathfrak{g}$.   This formula has the shape
\[ J(f_1, f_2) = J(\hat{f_1}, \check{f_2}) \]
where $J(\phi_1, \phi_2)$ is an expression written in terms of certain (orbital) integrals of test functions $\phi_1$ and $\phi_2$ that are locally constant and compactly supported on $\mathfrak{g}$; both sides of this identity have the same general form.  This formula provides a technically simpler analog of Arthur's local trace formula that nevertheless exhibits much of its structure.  For example, one encounters weighted orbital integrals on $\mathfrak{g}$ and essentially the same weights.   In addition, it is also a powerful tool in harmonic analysis, and already in the original article, \cite{Wal95}, Waldspurger applies this identity to extend earlier work of Harish-Chandra.  For these reasons, one would very much like to generalize this identity from reductive groups to reductive symmetric spaces.

The case of $F^{\times} \backslash SL_2(F)$ is interesting for two reasons:  first, it provides one of the simplest concrete analogs of the local trace formula, having only a handful of terms, each with fairly transparent structure;  second, it provides a template for the more general relative Lie algebra trace formula, in which one can already see the need for a second truncation procedure, which supplements Arthur's techniques.  One might reasonably expect valuable insight into the development of the relative trace formula on the group from this procedure.

We will work over a $p$-adic field $F$ of residue characteristic greater than two and  characteristic zero.  In order to state this special case, we need some notation.  Let $\mathcal{O}$ be the ring of integers of $F$, let $\nu_{\varpi}$ be a valuation of $F$, and fix a uniformizing element $\varpi$. Normalize the valuation by
$\nu_\varpi(\varpi)=1$ and set $\nu_\varpi(0)=+\infty$.  Write
$q=|\mathcal O/\varpi\mathcal O|$,
$|x|=q^{-\nu_\varpi(x)}$ for $x\ne0$, $G=SL_2(F)$, $K=SL_2(\mathcal O)$, and
\[
 H=\left\{
 \begin{pmatrix}t&0\\0&t^{-1}\end{pmatrix}
 :t\in F^\times
 \right\}
\]

Let $\mathcal{K}$ denote a set of representatives of the fourth-power classes of $F^{\times}$, which we assume to contain 1, and let $\mathcal{K}' := \mathcal{K} \cap (F^{\times})^2$.   Set
\[ \mathfrak{t}_{\gamma} := \left\{ \begin{pmatrix} & b\gamma \\ b & \end{pmatrix} : b \in F \right\} \]
and $T_{\gamma} = Z_G(\mathfrak{t}_{\gamma})$, the centralizer of $\mathfrak{t}_{\gamma}$ in $G$.  Let $A_{\gamma}$ be the $F$-split part of the torus $T_{\gamma}$.   Set
\[ \mathfrak{h} := \left\{ \begin{pmatrix} h & \\ & -h \end{pmatrix} : h \in F \right\} \text{ and }  \mathfrak{h}_{\epsilon} := \left\{ \begin{pmatrix} h & \\ & -h \end{pmatrix} : \nu_{\varpi}(h) \geq \nu_{\varpi}(\epsilon) \right\} \]
for a new truncation parameter in $F^{\times}$ that we call $\epsilon$.  The complement of $\mathfrak{h}_{\epsilon}$ in $\mathfrak{h}$ will be written $\mathfrak{h}_{\epsilon}^c$.   We will also write
\[ \tilde{f}(X) := \int_{K} \hat{f}(k^{-1}Xk) \, dk \]
using a Haar measure $dk$ on $K$, normalized so that $K$ has measure one.
Note that in every nilpotent integral over $F^\times$, the differential $dn$
will denote the restricted additive Haar measure on $F$, normalized
so that $\operatorname{vol}(\mathcal O)=1$.  All quotient measures will be induced by the specified
Haar measures.

Here is the formula we will calculate:

\begin{prop}
\label{prop:intro-trace-formula}
Let $f$ be a locally constant, compactly supported function on the Lie
algebra $\mathfrak{sl}_2$ of trace zero matrices.  Then there is an
integer $N_f$ such that, for every $\epsilon\in F^\times$ with
$\nu_\varpi(\epsilon)\geq N_f$, there is an integral identity,
\begin{align*}
	 \sum_{\gamma \in \mathcal{K} \setminus \mathcal{K}'}  \frac{1}{|\mu_4(F)|} &\int_{\mathfrak{t}_{\gamma}}  |\gamma b|  \int_G  f\left(g^{-1} \begin{pmatrix} & b\gamma \\ b & \end{pmatrix} g \right)\, dg\, db \\ & \hspace{8ex} +   \frac{1}{2}  \int_{\mathfrak{t}_1} |b| \int_{A_1 \backslash G}  f\left(g^{-1} \begin{pmatrix} & b \\ b & \end{pmatrix} g \right)\,   (1 - 2 H_B(g) + 2 H_{\bar{B}}(g))\, dg \, db \\
& = \int_{F^\times}  \tilde{f} \begin{pmatrix} 0 & 2n \\ & 0 \end{pmatrix} \, ( 1 + \nu_{\varpi}(n) - \nu_{\varpi}(\epsilon))\, dn + \frac{1}{1 - |\varpi|} \int_{\mathfrak{h}_{\epsilon}^c} \int_{H \backslash G}  \hat{f}(g^{-1}Xg)\, dg\, dX
\end{align*}

In this expression, $\mu_4(F)$ is the group of fourth roots of unity in $F$; the Haar measures on $G$, $H$, and $A_1$ give $K$, $H\cap K$, and $A_1\cap K$ volume one; quotient
measures are induced from these normalized measures.  Conventions normalizing the Fourier transform and additive measures are laid out in
section~\ref{2}; the Borels $B$ and $\bar B$, and the functions 
$H_B$ and $H_{\bar B}$ are fixed as in section~\ref{3} and subsection~4.3.  The measures on the Cartan subalgebras are coordinatewise measures, and integrals over them are understood to be 
taken over their regular elements, so the point $b=0$ may be omitted, for example.  All
integrals in the identity are absolutely convergent.
\end{prop}

This is the analog of the local trace formula for the symmetric space $F^{\times} \backslash SL_2(F)$ at the level of Lie algebras.  Notice that there appears on the right hand side an integral over the truncated set $\mathfrak{h}_{\epsilon}^c$.  This is because that particular integral need not converge if taken over the entire Lie algebra $\mathfrak{h}$.  Instead, we have had to remove the potentially divergent part, and absorb it into the weight factor in the first term on the right hand side.  Note that the two terms on the right do depend on $\epsilon$, while their sum does not when
$\nu_\varpi(\epsilon)\geq N_f$.

Because the set of fixed points of the involution $\theta$ is so simple in this case, it is not very difficult to truncate the set $\mathfrak{h}$ directly, so that it generates explicitly a weighted nilpotent integral.

\subsection{2026 update}

Since the original version of this paper, Delorme, Harinck, and Souaifi 
developed the geometric side for symmetric pairs associated with 
unramified quadratic extensions, and Delorme and Harinck proved 
a formula for $PGL_2(E)$ relative to $PGL_2(F)$ \cite{DHS19,DH17}.
The calculation of the $\theta$-split side in section~\ref{4} is the
rank-one precursor to the author's later general treatment \cite{Spa09}.
Also pertinent, Li derived an invariant local trace formula for Guo--Jacquet
symmetric spaces \cite{Li23}, and the global Guo--Jacquet trace
formula of Chaudouard and Li relates singular geometric terms to 
nilpotent contributions from other infinitesimal trace formulas \cite{CL26}.  
In this richer context, the present computation provides an early,
completely explicit rank-one Lie-algebra model in which the need for some
additional truncation or regularization procedure is manifest.

This revision stays faithful to the structure and target theorem of the original
preprint.  Intermediate sign errors have been fixed, measure and normalization
conventions have been made explicit, and the convergence arguments have been 
clarified.

\subsection{Outline}

In section \hyperref[2]{2}, we set the stage for this formula more explicitly.  In section \hyperref[3]{3}, we pin down the first truncation procedure, which requires the choice of a compactly supported and locally constant function on the symmetric space.  In section \hyperref[4]{4} and section \hyperref[5]{5}, we develop one side of the trace formula and then the other.  But we would like to draw attention to section \hyperref[5]{5}, in which one sees for the first time a second truncation procedure which remedies a form of divergence that does not appear in the Lie algebra analog of the original local trace formula.  Finally, in section \hyperref[6]{6}, we will combine these developments into a pair of identities, one for the symmetric space $F^{\times} \backslash SL_2(F)$ and another that follows from abelian Fourier analysis.

\subsection{Acknowledgements}

The author is greatly indebted to Professor Robert E. Kottwitz for suggesting the relative trace formula as a research topic, and for his unparalleled support and guidance as a thesis advisor during the author's years at the University of Chicago.

\section{Context}

\label{2}

The $F$-points $G$ of the algebraic group $\mathbb{G} := SL_2$ can be written in the following way:
\[  SL_2(F) = \left\{ \begin{pmatrix} a & b \\ c & d \end{pmatrix} : ad - bc = 1 \right\} \subset GL_2(F) \]
The involution
\[ \theta := \text{Inn\,} \begin{pmatrix} 1 & 0 \\ 0 & -1 \end{pmatrix} \]
on $GL_2(F)$ preserves $SL_2(F)$ and is defined over $\mathcal{O}$.   As a result, it preserves the maximal compact subgroup $K := \mathbb{G}(\mathcal{O})$ and acts on the quotient $G / K$.   The fixed points of $\theta$ in $G$ will be written $H$.  For this involution, $H$ is the set of diagonal matrices with determinant 1:
\[ \left\{ \begin{pmatrix} a &  \\  & a^{-1} \end{pmatrix} \right\} \subset SL_2(F) \]
The abelian group $H$ is an $F$-split maximal torus of $SL_2(F)$.
 
Let $\mathfrak{g} = \mathfrak{sl}_2$ be the Lie algebra of $SL_2(F)$ and let $\mathfrak{h}$ be the Lie algebra of $H$, so that $\mathfrak{h}$ is a subalgebra of $\mathfrak{g}$.   The Killing form provides an orthogonal decomposition of $\mathfrak{g}$:
\[ \mathfrak{g} = \mathfrak{h} \oplus \mathfrak{h}^{\perp} \]
where $\mathfrak{h}^{\perp}$ is a vector subspace of $\mathfrak{g}$ that is not a subalgebra.    We can express this vector space explicitly in coordinates:
\[ \mathfrak{h}^{\perp} = \left\{ \begin{pmatrix}  & b \\ c & \end{pmatrix} \right\} \subset \mathfrak{g} \]
Note that this set identifies in a natural way with the tangent space at the identity coset of the symmetric space $H \backslash G$.   

Given an additive character $\psi$ of $F$ that is trivial on $\mathcal{O}$ but not on $\varpi^{-1}\mathcal{O}$, a Killing form $\kappa(X,Y)=4\operatorname{tr}(XY)$, and an additive Haar measure of $F$ with $\text{vol}(\mathcal{O}) = 1$ (which induces coordinate product measures on $\mathfrak{g}$, $\mathfrak{h}$, and $\mathfrak{h}^{\perp}$) one can define the Fourier transform
\[ \hat f(Y)=\int_{\mathfrak g} f(X)\,\psi(\kappa(X,Y))\,dX \]
Because the residue characteristic of $F$ is greater than two, the usual
coordinate lattices in $\mathfrak g$, $\mathfrak h$, and
$\mathfrak h^\perp$ are self-dual. As a result, the Fourier normalization constants here equal 1.

Let $f$ be a locally constant, compactly supported function on the Lie algebra $\mathfrak{g}$, and let $C_c^{\infty}(\mathfrak{g})$ denote the set of such functions.  The starting point for the relative Waldspurger trace formula is this variant of the Plancherel identity:
\[ \int_{\mathfrak{h}} \hat{f}(X)\, dX = \int_{\mathfrak{h}^{\perp}} f(X)\, dX \]
Because the Fourier transform commutes with the action induced by conjugation of $G$ on $C_c^{\infty}(\mathfrak{g})$, we can adjust this identity so that each side is a function on $G$:
\[ \int_{\mathfrak{h}} \hat{f}(g^{-1}Xg)\, dX = \int_{\mathfrak{h}^{\perp}} f(g^{-1}Xg)\, dX \]
The adjoint action of $H$ preserves $\mathfrak{h}$ and the Killing form, which means that it necessarily preserves $\mathfrak{h}^{\perp}$ as well.   We can therefore interpret each side of this identity as a function on the noncompact space $H \backslash G$.   After introducing a compactly supported function $\phi(g)$ to guarantee convergence, we can integrate both sides of this identity over $H \backslash G$ to obtain
\[ \int_{H \backslash G} \phi(g) \int_{\mathfrak{h}} \hat{f}(g^{-1}Xg)\, dX\, dg = \int_{H \backslash G} \phi(g) \int_{\mathfrak{h}^{\perp}} f(g^{-1}Xg)\, dX\, dg \]
The bulk of the development of the trace formula consists in expressing both sides of this identity in terms of possibly weighted orbital integrals.   

We proceed in three stages.   First, we will choose the function $\phi(g)$ so that it correctly extends Arthur's truncation procedure.   Second, we will develop the right-hand side of this identity, using techniques of Arthur and Harish-Chandra.   Third, we will compute the left-hand side explicitly, by breaking up the integral over $\mathfrak{h}$ according to a second truncation parameter $\epsilon$.

\section{The choice of \texorpdfstring{$\phi$}{phi}}

\label{3}

Let $A$ denote the torus
\[ \left\{ \begin{pmatrix} a & b \\ b & a \end{pmatrix} : a^2 - b^2 = 1 \right\}  \]
Because the characteristic polynomial of any element in this torus splits over $F$, this is an $F$-split torus of dimension 1.  Because $G$ has rank 1, this torus is also maximal.  

Identify $F^\times$ with $A$ using this map:
\[
 a(t)=\frac12
 \begin{pmatrix}
 t+t^{-1}&t-t^{-1}\\
 t-t^{-1}&t+t^{-1}
 \end{pmatrix}
\]
Let $B$ be the Borel subgroup containing $A$ that becomes upper triangular after changing basis by conjugation with this element:
\[
	 C= \begin{pmatrix}1&1\\1&-1\end{pmatrix} \in GL_2(\mathcal{O})
\]
Then 
\[ X_*(A)_{dom} \cong \mathbb{Z}_{\geq 0} \]
where $X_*(A)_{dom}$ denotes the set of coweights of $A$ that are positive with respect to the Borel subgroup $B$.   In this context, we can define a map given by the Cartan decomposition:
\[ \text{Cartan\,} : G \twoheadrightarrow K \backslash G / K \cong X_*(A)_{dom} \cong \mathbb{Z}_{\geq 0} \]
This map is proper, and so the preimage of the union of a finite number of nonnegative integers is always a compact set.

We choose a nonnegative integer $\mu$, which is called the truncation parameter.   Following Arthur's truncation procedure for his local trace formula, we associate to this number a truncation function
\[ \bar{\omega}(g, \mu) := \begin{cases} 1 & \text{Cartan\,} \theta(g)^{-1} g \leq \mu \\ 0 & \text{otherwise} \end{cases} \]
A technical point is that the function
\[ \tau : g \mapsto \theta(g)^{-1}g \]
sends $H \backslash G$ into $G$ and is a closed immersion (as in \cite{Ric82}), which implies that $\bar{\omega}$ is compactly supported and locally constant as a function on the symmetric space $H \backslash G$.   It is this function that we use to truncate our trace formula: $\phi(g) = \bar{\omega}(g, \mu)$.

\section{The \texorpdfstring{$\theta$}{theta}-split Side}

\label{4}

\subsection{The Weyl integration formula for
\texorpdfstring{$\mathfrak{h}^{\perp}$}{h-perp}}

In this subsection, we provide an integration formula that replaces an integral over $\mathfrak{h}^{\perp}$ with an iterated integral: the first over the $H$-orbits of elements in $\mathfrak{h}^{\perp}$ and the second over a chosen set of representatives of these orbits.   

Almost every element of $\mathfrak{h}^{\perp}$ can be written in the following form:
\[ \begin{pmatrix} & a\gamma \\  a& \end{pmatrix} \]
By conjugating this matrix by elements of $H$, we can adjust $\gamma$ so that it belongs to a chosen set of representatives $\mathcal{K}$ for the fourth-power-classes of $F^{\times}$.  Indeed, diagonal conjugation changes the ratio of the two off-diagonal entries by a fourth power.  In other words, off a set of measure 0, every element of $\mathfrak{h}^{\perp}$ is $H$-conjugate to an element in
\[ \mathfrak{t}_{\gamma} :=  \left\{ \begin{pmatrix} & a\gamma \\  a& \end{pmatrix} : a \in F \right\} \]
for some $\gamma \in \mathcal{K}$.  Let $\mu_4(F)$ be the set of fourth roots of unity in $F$.  Then the integral over $\mathfrak{h}^{\perp}$ is
\begin{align*}
  \int_{\mathfrak{h}^{\perp}} f(g^{-1}Xg)\, dX &= \int_{F^2} f\left(g^{-1} \begin{pmatrix} & b \\ c & \end{pmatrix} g \right) \, db\, dc \\ &= \sum_{\gamma \in \mathcal{K}}  \frac{1}{|\mu_4(F)|} \int_{F^2} |\gamma b| \, f \left(g^{-1} \begin{pmatrix} & ba^{-2} \gamma \\ ba^{2} & \end{pmatrix} g \right)\, \frac{da}{|a|}\, db \\
  &= \sum_{\gamma \in \mathcal{K}} \frac{1}{|\mu_4(F)|} \int_{F^2} |\gamma b| \, f \left(g^{-1} \begin{pmatrix} a & \\ &  a^{-1} \end{pmatrix}^{-1} \begin{pmatrix}  & b \gamma \\ b &  \end{pmatrix} \begin{pmatrix} a & \\ & a^{-1} \end{pmatrix} g \right)\, \frac{da}{|a|}\, db
\end{align*}
where we allow the integrand to be undefined on a set of measure 0.  This can be written in the following way:
\[  \int_{\mathfrak{h}^{\perp}} f(g^{-1}Xg)\, dX = \sum_{\gamma \in \mathcal{K}}  \frac{1 - |\varpi|}{|\mu_4(F)|} \int_{\mathfrak{t}_{\gamma}}  |\gamma b| \int_H f \left(g^{-1} h^{-1} \begin{pmatrix} & b \gamma \\ b & \end{pmatrix} h g \right)\,dh\, db \]
where
\[ dh = \frac{1}{1 - |\varpi|} \frac{da}{|a|} \]
has been normalized to give $H \cap K$ measure one.   

Because this is a special case of the Weyl integration formula for
reductive symmetric spaces, we will refer to this identity as the Weyl
integration formula; for the nonarchimedean setting see \cite{RR96}, and
for general symmetric-space background see \cite{Hel78}.

\subsection{The general expansion}

For the choice of truncation function made in section \hyperref[3]{3}, we will consider the expression
\[  \int_{H \backslash G} \bar{\omega}(g, \mu) \int_{\mathfrak{h}^{\perp}} f(g^{-1}Xg)\, dX\, dg \]
which we call the $\theta$-split side.    Our goal will be to develop this expression into a sum of (orbital) integrals, each integrated over certain Cartan subalgebras of $\mathfrak{g}$ that contain only elements in $\mathfrak{h}^{\perp}$.

Expanding the integral over $\mathfrak{h}^{\perp}$ with the Weyl integration formula, introducing the truncation function $\bar{\omega}(g, \mu)$, and integrating this expression over $H \backslash G$ yields another expression for the $\theta$-split side:
\begin{align*}
	\sum_{\gamma \in \mathcal{K}}\frac{1 - | \varpi |}{|\mu_4(F)|}  \int_{\mathfrak{t}_{\gamma}}   |\gamma b| &\int_{G}   f \left(g^{-1}  \begin{pmatrix} & b \gamma \\ b & \end{pmatrix} g\right)\, \bar{\omega}(g, \mu) \, dg \, db \\
		&= \sum_{\gamma \in \mathcal{K}} \frac{1 - | \varpi |}{|\mu_4(F)|}  \int_{\mathfrak{t}_{\gamma}} |\gamma b| \int_{A_{\gamma} \backslash G}  f \left(g^{-1}  \begin{pmatrix} & b \gamma \\ b & \end{pmatrix} g\right)\, \int_{A_{\gamma}} \bar{ \omega}(ag, \mu) \, da\, dg \, db
\end{align*}
where $A_{\gamma}$ is the maximal $F$-split subtorus of $T_{\gamma} := Z_G(\mathfrak{t}_{\gamma})$.

\begin{definition}
The inner integral
\[ \omega_{\gamma}(g, \mu) := \int_{A_{\gamma}} \bar{\omega}(ag, \mu) \, da = \text{meas\,}_{A_{\gamma}} \{ a : \text{Cartan\,} \tau(ag) \leq \mu \} \]
is called the weight factor (associated to $\gamma$).   
\end{definition}

Notice that when $\gamma$ is not a square, the torus $T_{\gamma}$ is anisotropic over $F$, and so $A_{\gamma}$ is trivial (and given measure one).  In this case, the weight factor associated to $\gamma$ is simply $\bar{\omega}(g, \mu)$.    When $\gamma$ is a square, however, these weight factors are slightly more complicated.

Suppose that $\gamma$ belongs to $\mathcal{K}' = \mathcal{K} \cap (F^{\times})^2$.   We start by relating $A_{\gamma}$ to the maximal $F$-split torus $A = A_1$, where $1$ represents the trivial fourth-power-class of $F^{\times}$.  With a square-root $\sqrt{\gamma} \in F^{\times}$, we can set
\begin{align*}
	r_{\gamma} := \begin{pmatrix} \frac{1+\sqrt{\gamma}}{2} & \frac{\sqrt{\gamma}-1}{2} \\ \frac{\sqrt{\gamma} - 1}{2\sqrt{\gamma}} & \frac{1 + \sqrt{\gamma}}{2\sqrt{\gamma}} \end{pmatrix}
\end{align*}
so that
\begin{align*}
	r_{\gamma}^{-1} A_{\gamma} r_{\gamma} = A
\end{align*}
and 
\begin{align*}
	\tau(r_{\gamma}) = \theta(r_{\gamma})^{-1} r_{\gamma} = \begin{pmatrix} \frac{\gamma + 1}{2\sqrt{\gamma}} & \frac{\gamma - 1}{2\sqrt{\gamma}} \\ \frac{\gamma - 1}{2\sqrt{\gamma}} & \frac{\gamma + 1}{2 \sqrt{\gamma}}\end{pmatrix} \in A
\end{align*}
which has image $|\nu_\varpi(\sqrt\gamma)|$ under the projection $A \rightarrow X_*(A)_+$. We equip $A_\gamma$ with the measure transported from $A$ by conjugation with $r_\gamma$.

We consider the following part of the $\theta$-split side:
\begin{align*}
 \sum_{\gamma \in \mathcal{K}'}\frac{1-|\varpi|}{|\mu_4(F)|}  \int_{\mathfrak{t}_{\gamma}} |\gamma b|   \int_{A_{\gamma} \backslash G}  f \left(g^{-1} \begin{pmatrix} & b \gamma \\ b & \end{pmatrix} g\right)\, \omega_{\gamma}(g, \mu)\, dg\, db
\end{align*}
The elements $r_{\gamma}$ allow us to change variables so that this expression becomes
\begin{align*}
	 	&\frac{1-|\varpi|}{|\mu_4(F)|}  \int_{\mathfrak{t}_1} | b | \int_{A \backslash G}  f \left(g^{-1} \begin{pmatrix} & b \\ b & \end{pmatrix} g\right)\, \left(  \sum_{\gamma \in \mathcal{K}'}  \, \omega_{\gamma}(r_{\gamma}g, \mu)\right) \, dg\, db
\end{align*}
(we have also changed variables by the map $b \mapsto b \gamma^{-\frac{1}{2}}$). The inner sum acts as a new weight factor, and we may write it in the following way:
\begin{align*}
	\sum_{\gamma \in \mathcal{K}'}  \omega_{\gamma}(r_{\gamma}g, \mu) &:= \sum_{\gamma \in \mathcal{K}'} \int_{A_{\gamma}} \bar{\omega}(ar_{\gamma}g, \mu) \, da = \sum_{\gamma \in \mathcal{K}'}  \int_A \bar{\omega}(r_{\gamma}ag, \mu)\, da
\end{align*}
Our goal for the remainder of this section will be to compute this sum explicitly. 

\subsection{Computation of the weight factor}

Recall that $\bar{\omega}(r_{\gamma}ag, \mu)$ is defined by the inequality $\text{Cartan\,} \tau(r_{\gamma}ag) \leq \mu$.   This condition can be computed explicitly, provided that $\mu$ is bigger than some integer that depends on $g$.

We first require some notation.   Let $g = tnk$ be the Iwasawa decomposition of $g$, where $t$ is in $A \cong F^{\times}$, $n$ is in the unipotent radical $N \cong F$ of $B$, and $k$ is in $K$.   Recall the function $H_B$ on $G$, which is defined to be $H_B(g) := \nu_{\varpi}(t)$; $H_{\bar B}$ is defined similarly using the opposite Borel and the same identification of $A$ with $F^\times$.  We will also set $d(g) := H_B(g) - H_{\bar{B}}(g)$.  For $a\in A$, notice that $H_B(ag)=\nu_\varpi(a)+H_B(g)$ and
$H_{\bar B}(ag)=\nu_\varpi(a)+H_{\bar B}(g)$, meaning that $d(ag)=d(g)$ descends to a function on $A\backslash G$.

\begin{lemma}\label{lem:rank-one-weight}
If $\mu\geq d(g)$, the inequality
\[
 \operatorname{Cartan}\tau(r_\gamma ag)\leq\mu
\]
is equivalent to
\[
 -2H_{\bar B}(g)-\mu
 \leq \nu_\varpi(\sqrt\gamma)+2\nu_a
 \leq \mu-2H_B(g)
\]
where $\nu_a = \nu_\varpi(a)$ denotes the signed valuation of $a\in A$.
\end{lemma}

\begin{proof}
First write
\begin{align*}
   \text{Cartan\,} \tau(r_{\gamma}ag) &= \text{Cartan\,} \left( \theta(n)^{-1} \tau(r_{\gamma}) a^2 t^2 n \right)
\end{align*}
This is easier to compute if we temporarily change basis as in section \hyperref[3]{3} so that $A$ is the set of diagonal matrices, and $B$ is upper triangular.  In these coordinates, write 
\[
	n=\begin{pmatrix}1&u\\0&1\end{pmatrix} \text{  and  } \theta(n)^{-1}=\begin{pmatrix}1&0\\-u&1\end{pmatrix}
\]
Notice that $d=d(g)=\max\{0,-\nu_\varpi(u)\}$, so $H_{\bar B}(g)=H_B(g)+\min\{0,\nu_\varpi(u)\}$.   Write
\[ \tau(r_{\gamma}) a^2 = \begin{pmatrix} s & \\ & s^{-1} \end{pmatrix} \text{  and  } t^2 = \begin{pmatrix} r & \\ & r^{-1}\end{pmatrix} \]
Then
\begin{align*}
 \text{Cartan\,}  \left( \theta(n)^{-1} \tau(r_{\gamma}) a^2 t^2 n \right) &= \text{Cartan\,} \begin{pmatrix} 1 &  \\ -u & 1 \end{pmatrix} \begin{pmatrix} sr & \\ & (sr)^{-1} \end{pmatrix}\begin{pmatrix} 1 & u \\ & 1 \end{pmatrix} \\
 &=  \text{Cartan\,} \begin{pmatrix} sr & sru \\ -sru & (sr)^{-1} - sru^2  \end{pmatrix} \\
 &= - \text{min\,} \{ \nu_{\varpi} (sr),  \nu_{\varpi}(sru),    \nu_{\varpi}((sr)^{-1} - sru^2) \}
\end{align*}

Write $x=\nu_\varpi(sr)$. If $d=0$, then $u\in\mathcal O$ and both unipotent matrices ($n$ and its opposite) lie in $K$, so the Cartan inequality simply asserts that $-\mu\le x\le\mu$. So suppose that $d > 0$ and $\nu_\varpi(u)=-d$. The first two arguments in the minimum have valuations $x$ and $x-d$; the third has summands with valuations $-x$ and $x-2d$. We consider two cases according to whether the valuations of these summands differ or coincide.

Case 1: these summands have different valuations. In this case, the claim that all three arguments in the minimum have valuation at least $-\mu$ is equivalent to the claim that $x\le\mu$ and $x-2d\ge-\mu$, which we can write more compactly: $2d-\mu\le x\le\mu$.

Case 2: these summands have the same valuation. In this case, $x = d$ and $\mu \geq d$, so both  valuations equal $-d \geq -\mu$. Since the first two arguments in the minimum have valuations $d$ and $0$ respectively, we can always bound the Cartan by $\mu$. One can check $2d-\mu\le x\le\mu$ is always true in this case as well.

The Cartan inequality is therefore equivalent to $2d-\mu\le x\le\mu$. Substituting $x=\nu_\varpi(\sqrt{\gamma})+2\nu_a+2H_B(g)$ and $d(g)=H_B(g)-H_{\bar B}(g)$ then gives the sought inequality:
\[
-2H_{\bar B}(g)-\mu\le\nu_\varpi(\sqrt{\gamma})+2\nu_a\le\mu-2H_B(g)
\]
as claimed.

\end{proof}

In other words, this lemma implies that if $A$ is equipped with the Haar measure that assigns $A(\mathcal{O}) = A \cap K$ a measure of one, then
\begin{align*}
	\int_A \bar{\omega}(r_{\gamma}ag, \mu)\, da &=  \# \{ \nu \in \mathbb{Z} : - 2 H_{\bar{B}}(g) - \mu \leq \nu_{\varpi}(\sqrt{\gamma}) + 2 \nu \leq \mu - 2 H_B(g) \}
\end{align*}
when $\mu$ is sufficiently large in a sense that depends on $g$.  We would next like to sum this right hand side over the elements of $\mathcal{K}'$.

Because $F$ is a $p$-adic field, we can assume that $\mathcal{K}'$ has been chosen to contain $| \mu_4(F) |$ elements, one half of them having valuation zero and the other half having valuation two.   Then 
\begin{align*}
	\sum_{\gamma \in \mathcal{K}'} \, \omega_{\gamma}(r_{\gamma}g, \mu) &= \frac{| \mu_4(F)|}{2} \, \# \{ \nu : - 2 H_{\bar{B}}(g) - \mu \leq 2 \nu \leq \mu - 2 H_B(g) \} \\
	& \hspace{10ex} +  \frac{|\mu_4(F)|}{2}   \,  \# \{ \nu : - 2 H_{\bar{B}}(g) - \mu \leq 1+ 2 \nu \leq \mu - 2 H_B(g) \}\\
	&=  \frac{|\mu_4(F)|}{2}  (1 + 2 \mu - 2 H_B(g) + 2 H_{\bar{B}}(g))
\end{align*}
This expression is only valid when $\mu$ is sufficiently large, but we may nevertheless use it to form an approximation to the $\theta$-split side.

\begin{definition}
Set $J_{-}(f, \mu)$ equal to the expression
\begin{align*}
	\frac{1 - |\varpi|}{2}  \int_{\mathfrak{t}_1} |b| \int_{A_1 \backslash G}&  f\left(g^{-1}\begin{pmatrix} & b \\ b & \end{pmatrix} g\right)\,  \left( 1 + 2 \mu - 2 H_B(g) + 2 H_{\bar{B}}(g) \right) \, dg\, db \\
	 & +  \sum_{\gamma \in \mathcal{K} \setminus \mathcal{K}'} \frac{1 - |\varpi|}{|\mu_4(F)|} \int_{\mathfrak{t}_{\gamma}} |\gamma b | \int_G  f\left(g^{-1}\begin{pmatrix} & b \gamma \\ b & \end{pmatrix} g\right)\, dg\, db
\end{align*}
\end{definition}

We prove next that the difference between this approximation and the original expression tends to 0 as $\mu$ approaches infinity, so that this is the ``asymptotic" polynomial in $\mu$ to the $\theta$-split side near $\mu = \infty$.

\begin{prop}
For any locally constant, compactly supported function $f$ on the Lie algebra $\mathfrak{g}$,
\[ \lim_{\mu \rightarrow \infty} \left( \text{$\theta$-split side} - J_{-}(f, \mu) \right) = 0 \]
\end{prop}

\begin{proof}
When $\gamma$ is not a square, bound the magnitude of the difference
\[ \left| |\gamma b|  \int_{A_{\gamma} \backslash G}  f\left(g^{-1}\begin{pmatrix} & b\gamma \\ b & \end{pmatrix} g\right) \, \left( 1 - \omega_{\gamma}(g, \mu) \right) \, dg \right| \leq  |\gamma b|  \int_{A_{\gamma} \backslash G} \left| f\left(g^{-1}\begin{pmatrix} & b\gamma \\ b & \end{pmatrix} g\right) \right| \, dg \]
Up to a constant depending only on $\gamma$ and the measure
normalizations, this is a normalized orbital integral
of $|f|$.  Indeed, $T_\gamma$ is compact and $A_\gamma$ is trivial, so the
integral over $G$ is $\operatorname{vol}(T_\gamma)$ times the integral
over $T_\gamma\backslash G$.
We can assume the $b$ lie in a compact set since
$\det\begin{pmatrix}&b\gamma\\b&\end{pmatrix}=-\gamma b^2$ is conjugation invariant.
Harish-Chandra's boundedness theorem for normalized orbital integrals
\cite[Theorem~13]{HC70}, together with compact support, implies that this
bound is integrable on $\mathfrak{t}_{\gamma}$.

When $\gamma$ is a square, one first observes that when $\mu \geq d(g)$,
\[ W(g, \mu) := \sum_{\gamma \in \mathcal{K}'} \omega_{\gamma}(r_{\gamma} g, \mu)  = \frac{|\mu_4(F)|}{2} (1 + 2 \mu - 2 H_B(g) + 2 H_{\bar{B}}(g))  = \frac{|\mu_4(F)|}{2} (1 + 2 \mu - 2 d(g)) \]
For $\mu<d(g)$, monotonicity gives
\[
0\leq W(g,\mu)\leq W(g,d(g))=\frac{|\mu_4(F)|}{2}
\]
Note that the polynomial $\frac{|\mu_4(F)|}{2}(1+2\mu-2d(g))$ is negative when
$\mu<d(g)$, since $\mu$ and $d(g)$ are integers.  It follows for every
nonnegative integer $\mu$ that
\[
\left|W(g,\mu)-\frac{|\mu_4(F)|}{2}(1+2\mu-2d(g))\right|
\leq |\mu_4(F)|d(g)
\]
and so we can bound the magnitude of the difference
\[ \int_{\mathfrak{t}_1} |b|  \int_{A \backslash G}  f\left(g^{-1}\begin{pmatrix} & b \\ b & \end{pmatrix} g\right) \, \left( \frac{|\mu_4(F)|}{2} (1 + 2 \mu- 2 H_B(g) + 2 H_{\bar{B}}(g)) - W(g, \mu) \right) \, dg\, db\]
by a weighted normalized orbital integral:
\[\int_{\mathfrak{t}_1} |b|  \int_{A \backslash G} \left| f\left(g^{-1}\begin{pmatrix} & b \\ b & \end{pmatrix} g\right) \right| \,  |\mu_4(F)|  d(g) \,  dg\, db \]
The integrability of this expression follows from Harish-Chandra's bound \cite[Theorem~13]{HC70} and
Waldspurger's logarithmic growth estimate for Arthur's weights
\cite[Lemme~III.5]{Wal95}.  More precisely, with
$X_b=\begin{pmatrix}&b\\b&\end{pmatrix}$, we see that for
$b$ in the relevant compact set and some $N$,
\[
 |b|\int_{A\backslash G}|f(g^{-1}X_bg)|\,d(g)\,dg
 \ll_f (1+|\nu_\varpi(b)|)^N
\]
Since $\sum_{m\geq0}(1+m)^Nq^{-m}<\infty$, this is locally integrable
in $b$.

That both these differences are bounded by absolutely integrable functions that do not depend on $\mu$ means that we can use the dominated convergence theorem to study their behavior as $\mu$ tends to infinity.  Specifically, the limit of the difference between the $\theta$-split side and this approximation is a linear combination of integrals of the form
\[ \int_{\mathfrak{t}_1} |b|  \int_{A \backslash G}  f\left(g^{-1}\begin{pmatrix} & b \\ b & \end{pmatrix} g\right) \, \lim_{\mu \rightarrow \infty} \left( W(g, \mu) -  \frac{|\mu_4(F)|}{2}(1 + 2 \mu- 2 H_B(g) + 2 H_{\bar{B}}(g))  \right) \, dg\, db = 0 \]
or, for $\gamma\in\mathcal K\setminus\mathcal K'$, of the form
\[ \int_{\mathfrak{t}_\gamma} |\gamma b|  \int_{A_\gamma \backslash G}  f\left(g^{-1}\begin{pmatrix} & b \gamma \\ b & \end{pmatrix} g \right) \, \lim_{\mu \rightarrow \infty} \left( 1 - \omega_{\gamma}(g, \mu) \right) \, dg\, db = 0 \]
which implies that the entire limit equals 0.
\end{proof}

\smallskip
\noindent\textit{L'esprit de l'escalier}
For a simpler proof, one may take $N=1$ above. Choose $R\geq0$ with
$\operatorname{supp}(f)\subseteq
\varpi^{-R}\mathfrak{sl}_2(\mathcal O)$.
In the coordinates of Lemma~\ref{lem:rank-one-weight},
writing $g=tn(u)k$, $f(g^{-1}X_bg) \neq 0$ implies that $|bu|\leq q^R$, so
\[
 d(g)=\max\{0,-\nu_\varpi(u)\}
 \leq R+|\nu_\varpi(b)|
\]
The necessary bound then follows immediately from the unweighted normalized
orbital-integral bound.

\section{The \texorpdfstring{$\theta$}{theta}-fixed side}

\label{5}

We now consider the left-hand side of the trace formula,
\[  \int_{H \backslash G} \bar{\omega}(g, \mu) \int_{\mathfrak{h}} \hat{f}(g^{-1}Xg)\, dX\, dg =   \int_{\mathfrak{h}} \int_{H \backslash G}  \hat{f}(g^{-1}Xg)\, \bar{\omega}(g, \mu)\, dg\, dX  \]
which we call the $\theta$-fixed side.  Unfortunately, this side is not amenable to Arthur's techniques: the inner integral would have to be replaced by an unweighted orbital integral which is in general not locally integrable on $\mathfrak{h}$ near the origin. But in the case of $SL_2(F)$, it is simple enough that we may just compute an expansion directly.

First, we divide $\mathfrak{h}$ into two sets, which depend on the choice of a second truncation parameter $\epsilon \in F^{\times}$.
\begin{align*}
	\mathfrak{h}_{\epsilon} &:= \left\{ \begin{pmatrix} x & \\ & -x \end{pmatrix} : | x | \leq | \epsilon | \right\} \text{ and } \mathfrak{h}^c_{\epsilon} := \{ X \in \mathfrak{h} : X \notin \mathfrak{h}_{\epsilon} \}
\end{align*} 
The $\theta$-fixed side can then be broken into two terms:
\[ \int_{\mathfrak{h}_{\epsilon}} \int_{H \backslash G}  \hat{f}(g^{-1}Xg)\, \bar{\omega}(g, \mu)\, dg\, dX  + \int_{\mathfrak{h}^c_{\epsilon}} \int_{H \backslash G}  \hat{f}(g^{-1}Xg)\, \bar{\omega}(g, \mu)\, dg\, dX \]
For the second term, $\bar{\omega}(g, \mu)$ itself will act as the weight factor.   For the first term, we must derive an ad hoc germ expansion. 

Because $H$ is a maximal torus in $G = SL_2(F)$, we can apply the Iwasawa decomposition to write $G = HN'K$, where $N'$ is the unipotent radical of the upper triangular Borel subgroup $B'$ that contains $H$.  If we set
\[ \tilde{f}(X) := \int_K \hat{f}(k^{-1}Xk)\, dk \]
then the first term can be written
\[ \int_{\mathfrak{h}_{\epsilon}} \int_{N'}  \tilde{f}(n^{-1}Xn)\, \bar{\omega} \left( n, \mu\right)\, dn\, dX \]
because $\bar{\omega}$ is right $K$-invariant and left $H$-invariant.   In coordinates, this equals
\[ \int_{\{h \in F : | h | \leq| \epsilon| \}} \int_F  \tilde{f} \left(\begin{pmatrix} 1 & n \\ & 1 \end{pmatrix}^{-1} \begin{pmatrix} h & \\ & -h \end{pmatrix} \begin{pmatrix} 1 & n \\ & 1 \end{pmatrix} \right)\, \bar{\omega}(n, \mu)\, dn\, dh \]
where the measures $dn$ and $dh$ are both the additive measures on $F$.   This integral can be simplified to
\begin{align*}
 \int_{\{h \in F : | h | \leq | \epsilon| \}} \int_F  \tilde{f} \begin{pmatrix} h & 2hn \\ & -h \end{pmatrix} \, \bar{\omega}\left( n, \mu\right) \, dn\, dh
\end{align*}
or
\begin{align*}
 \int_{\{h \in F : | h | \leq |\epsilon| \}} \int_F  \tilde{f} \begin{pmatrix} h & 2n \\ & -h \end{pmatrix} \, \bar{\omega}\left( \begin{pmatrix} 1 & \frac{n}{h} \\ & 1 \end{pmatrix}, \mu\right)
 \, \frac{dn}{|h|}\, dh
\end{align*}
Now, a compactly supported and locally constant function on
$\mathfrak g$ is invariant under translation by a sufficiently small
$\mathcal O$-lattice.  Consequently, there is an integer $N_f$ such that
this lattice contains $\begin{pmatrix}h&\\&-h\end{pmatrix}$ whenever
$\nu_\varpi(h)\geq N_f$; for every $\epsilon\in F^\times$ with
$\nu_\varpi(\epsilon)\geq N_f$, this gives
\[ \tilde{f} \begin{pmatrix} h & 2n \\ & -h \end{pmatrix} =  \tilde{f} \begin{pmatrix}0 & 2n \\ & 0 \end{pmatrix} \]
for every $n\in F$ and every $h$ with $|h|\leq|\epsilon|$.  Fix any
$\epsilon$ satisfying this inequality.

The points $h=0$ and $n=0$ may be omitted from the following integrals.  Fubini's theorem allows us to reverse the order of integration to express the first term on the $\theta$-fixed side as
\[ \int_{F^\times}  \tilde{f} \begin{pmatrix} 0 & 2n \\ & 0 \end{pmatrix} \, \int_{\{h \in F : | h | \leq| \epsilon| \}} \bar{\omega}\left( \begin{pmatrix} 1 & \frac{n}{h} \\ & 1 \end{pmatrix}, \mu\right) \frac{dh}{|h|} \, dn \]
This expression is an integral over the Lie algebra $\mathfrak{n}'$ of the unipotent radical $N'$ with weight
\begin{align*}
	\int_{\{h \in F : | h | \leq| \epsilon| \}} \bar{\omega}\left( \begin{pmatrix} 1 & \frac{n}{h} \\ & 1 \end{pmatrix}, \mu\right) \frac{dh}{|h|} &=  \int_{\{ h \in F : |n| \cdot | \varpi |^{\mu} \leq | h | \leq |\epsilon| \}} \frac{dh}{|h|} \\
						&= \left( 1 - |\varpi| \right)
\text{max}\left(0, \nu_{\varpi}(n) + \mu - \nu_{\varpi}(\epsilon) + 1\right)
\end{align*}
for $n\neq0$.  This exact formula for the
weight factor provides an approximation to the $\theta$-fixed
side.

\begin{definition}
Set $J_{+}(\hat{f}, \mu) = J_{+}(\hat{f}, \mu, \epsilon)$ equal to
\begin{align*}
\left( 1 - |\varpi| \right) \int_{F^\times}  \tilde{f} \begin{pmatrix} 0 & 2n \\ & 0 \end{pmatrix} \, ( \nu_{\varpi}(n) + \mu - \nu_{\varpi}(\epsilon) + 1)\, dn +  \int_{\mathfrak{h}^c_{\epsilon}} \int_{H \backslash G}  \hat{f}(g^{-1}Xg)\, dg\, dX
\end{align*}
This expression is defined for all $\epsilon \in F^{\times}$, but the
local trace formula will hold only for $\epsilon$ satisfying
$\nu_\varpi(\epsilon)\geq N_f$.
\end{definition}

Again, we claim that the difference between this approximation and the original expression tends to 0 as $\mu$ approaches infinity.  This will imply that the polynomial approximations we have derived for the $\theta$-split side and the $\theta$-fixed side are equal, which is the statement of a working version of the relative trace formula for $F^{\times} \backslash SL_2(F)$.  

\begin{prop}
For any compactly supported, locally constant function $f$ on
$\mathfrak{g}$ and every $\epsilon\in F^\times$ such that
$\nu_\varpi(\epsilon)\geq N_f$,
\[ \lim_{\mu \rightarrow \infty} \left( \text{$\theta$-fixed side} - J_{+}(\hat{f}, \mu) \right) = 0 \]
\end{prop}

\begin{proof}
For $n\in F^\times$, when $\mu$ exceeds the valuation of $\epsilon$ minus the valuation of $n$ (and $\nu_\varpi(\epsilon) \geq N_f$),
\[ \int_{ \{ h \in F : |h | \leq| \epsilon| \}} \bar{\omega} \left( \begin{pmatrix} 1 & \frac{n}{h} \\ & 1 \end{pmatrix} , \mu \right) \frac{dh}{|h|} = (1-|\varpi|) (\nu_{\varpi}(n) + \mu - \nu_{\varpi}(\epsilon) + 1) \]
we can argue that the magnitude of the difference
\begin{align*}
	 & \tilde{f} \begin{pmatrix} 0 & 2n \\ & 0 \end{pmatrix} \, \left(  (1 - |\varpi|) ( \nu_{\varpi}(n) + \mu - \nu_{\varpi}(\epsilon) + 1)  - \int_{\{ h \in F : |h| \leq| \epsilon| \}}\bar{\omega} \left( \begin{pmatrix} 1 & \frac{n}{h} \\ & 1 \end{pmatrix}, \mu \right) \frac{dh}{|h|}  \right) \   
\end{align*}
is bounded by a constant multiple of
\begin{align*}
 & \left| \tilde{f} \begin{pmatrix} 0 & 2n \\ & 0 \end{pmatrix} \,    \right|
\end{align*}
which is absolutely integrable on $F$.  Indeed, compact support bounds
$\nu_\varpi(n)$ from below when $n \neq 0$ is in the support of $\tilde{f}$.

Because these differences are bounded by absolutely integrable functions that do not depend on $\mu$, we can apply the dominated convergence theorem to understand the behavior of this difference as $\mu$ tends to infinity.  Explicitly, the limit of the difference between the first term on the $\theta$-fixed side and the first term in its approximation $J_{+}(\hat{f}, \mu)$ is
\[\int_{F^\times} \tilde{f} \begin{pmatrix} 0 & 2n \\ & 0 \end{pmatrix} \, \lim_{\mu \rightarrow \infty}  \left( (1 - |\varpi|) ( \nu_{\varpi}(n) + \mu - \nu_{\varpi}(\epsilon) + 1) - \int_{\{h \in F : |h| \leq |\epsilon|\}} \bar{\omega} \left( \begin{pmatrix} 1 & \frac{n}{h} \\ & 1 \end{pmatrix}, \mu \right) \frac{dh}{|h|}  \right) \, dn \]
which equals zero.  Analogously, the difference between the second term on the $\theta$-fixed side and the second term in its approximation $J_{+}(\hat{f}, \mu)$ also tends to 0.  These limits, taken together, imply the limit of the proposition.
\end{proof}

\section{A relative trace formula for
\texorpdfstring{$F^{\times} \backslash SL_2(F)$}{F-times backslash SL2(F)}}

\label{6}

\Needspace{0.42\textheight}
\begin{thm}
\label{thm:relative-trace-formula}
For any locally constant, compactly supported function $f$ on the Lie
algebra $\mathfrak{g}=\mathfrak{sl}_2$, there is an integer $N_f$ such
that, for every $\epsilon\in F^\times$ with
$\nu_\varpi(\epsilon)\geq N_f$ and every
$\mu\in\mathbb Z_{\geq0}$, the following integral identity holds:
\begin{align*}
	 & \frac{1 - |\varpi|}{2} \int_{\mathfrak{t}_1} |b| \int_{A \backslash G}  f\left(g^{-1} \begin{pmatrix} & b \\ b & \end{pmatrix} g \right)\,  \left( 1 + 2 \mu - 2 H_B(g) + 2 H_{\bar{B}}(g) \right) dg \, db \\
	 & \hspace{30ex} + \sum_{\gamma \in \mathcal{K} \setminus \mathcal{K}'} \frac{1 - |\varpi|}{|\mu_4(F)|} \int_{\mathfrak{t}_{\gamma}}  |\gamma b|  \int_G  f\left(g^{-1} \begin{pmatrix} & b\gamma \\ b & \end{pmatrix} g \right)\, dg\, db  \\
&\hspace{5ex} = \left( 1 - |\varpi| \right) \int_{F^\times}  \tilde{f} \begin{pmatrix} 0 & 2n \\ & 0 \end{pmatrix} \, ( \nu_{\varpi}(n) + \mu - \nu_{\varpi}(\epsilon) + 1)\, dn +  \int_{\mathfrak{h}^c_{\epsilon}} \int_{H \backslash G}  \hat{f}(g^{-1}Xg)\, dg\, dX
\end{align*}
\end{thm}

\begin{proof}
This identity may be stated succinctly as $J_{-}(f, \mu) = J_{+}(\hat{f}, \mu)$.    We know that the $\theta$-fixed side equals the $\theta$-split side. By subtracting the limits
\begin{align*}
	\lim_{\mu \rightarrow \infty} (\text{$\theta$-fixed side } - J_{+}(\hat{f}, \mu)) &= 0 \\
	\lim_{\mu \rightarrow \infty} (\text{$\theta$-split side } - J_{-}(f, \mu)) &= 0
\end{align*}
one establishes this limit: 
\[\lim_{\mu \rightarrow \infty} (J_{+}(\hat{f}, \mu)-J_{-}(f, \mu)) = 0\]
Although they were obtained via approximations valid only for
large $\mu$, the polynomials $J_{+}(\hat{f},\mu)$ and
$J_{-}(f,\mu)$ are defined for all $\mu$.  Their difference tends to zero
as $\mu$ tends to infinity, so this polynomial is identically zero.
Therefore,
\[  J_{-}(f, \mu) =  J_{+}(\hat{f}, \mu) \]
for every $\mu \in \mathbb{Z}_{\geq 0}$, which is the stated identity.
\end{proof}

This theorem gives the Lie algebra relative trace formula for the symmetric space $F^{\times} \backslash SL_2(F)$, and one of the simplest non-trivial analogs of the local trace formula.   As an identity of polynomials in $\mu$, we may identify coefficients, obtaining two formulas:
\begin{equation}
\label{eq:abelian-relative-trace-formula}
\begin{aligned}
	 \int_{\mathfrak{t}_1} |b|  &\int_{A_1 \backslash G}  f \left(g^{-1} \begin{pmatrix} & b \\ b & \end{pmatrix} g \right)\, dg\, db = \int_F  \tilde{f} \begin{pmatrix} 0 & 2n \\ & 0 \end{pmatrix} \, dn
\end{aligned}
\end{equation}
and
\begin{equation}
\label{eq:relative-trace-formula-constant-term}
\begin{aligned}
	 \sum_{\gamma \in \mathcal{K} \setminus \mathcal{K}'}&  \frac{1}{|\mu_4(F)|} \int_{\mathfrak{t}_{\gamma}}  |\gamma b|  \int_G  f\left(g^{-1} \begin{pmatrix} & b\gamma \\ b & \end{pmatrix} g \right)\, dg\, db \\ & \hspace{10ex} + \frac{1}{2} \int_{\mathfrak{t}_1} |b| \int_{A_1 \backslash G}  f\left(g^{-1} \begin{pmatrix} & b \\ b & \end{pmatrix} g \right)\,   (1 - 2 H_B(g) + 2 H_{\bar{B}}(g))\, dg \, db \\
&= \int_{F^\times}  \tilde{f} \begin{pmatrix} 0 & 2n \\ & 0 \end{pmatrix} \, ( 1 + \nu_{\varpi}(n) - \nu_{\varpi}(\epsilon))\, dn + \frac{1}{1 - |\varpi|} \int_{\mathfrak{h}^c_{\epsilon}} \int_{H \backslash G}  \hat{f}(g^{-1}Xg)\, dg\, dX
\end{aligned}
\end{equation}

The first formula follows from ordinary (abelian) Fourier analysis. The second formula gives the trace formula for $F^{\times} \backslash SL_2(F)$, the integral identity of Proposition~\ref{prop:intro-trace-formula}.


\begin{thebibliography}{99}

\bibitem{Art91}
J.~Arthur.
\newblock {\em A Local Trace Formula}.
\newblock Publ. Math. Inst. Hautes \'Etudes Sci. 73 (1991), 5--96.

\bibitem{CL26}
P.-H.~Chaudouard and H.~Li.
\newblock {\em Une Formule des Traces pour les Espaces Sym\'{e}triques: Le
  Cas de Guo--Jacquet}.
\newblock Preprint, arXiv:2602.06617, 2026.

\bibitem{CPS90}
J.~W.~Cogdell and I.~Piatetski-Shapiro.
\newblock {\em The Arithmetic and Spectral Analysis of Poincar\'{e} Series}.
\newblock Perspectives in Mathematics, vol.~13, Academic Press, Boston, 1990.

\bibitem{DH17}
P.~Delorme and P.~Harinck.
\newblock {\em A Local Relative Trace Formula for $PGL(2)$}.
\newblock Pacific J. Math. 291 (2017), no.~1, 121--147.

\bibitem{DHS19}
P.~Delorme, P.~Harinck, and S.~Souaifi.
\newblock {\em Geometric Side of a Local Relative Trace Formula}.
\newblock Trans. Amer. Math. Soc. 371 (2019), no.~3, 1815--1857.

\bibitem{HC70}
Harish-Chandra.
\newblock {\em Harmonic Analysis on Reductive $p$-adic Groups}.
\newblock Notes by G.~van Dijk, Lecture Notes in Mathematics, vol.~162,
  Springer-Verlag, Berlin--Heidelberg, 1970.

\bibitem{Hel78}
S.~Helgason.
\newblock {\em Differential Geometry, Lie Groups, and Symmetric Spaces}.
\newblock Academic Press, New York, 1978.

\bibitem{Jac05}
H.~Jacquet.
\newblock {\em A Guide to the Relative Trace Formula}.
\newblock In {\em Automorphic Representations, $L$-Functions and Applications:
  Progress and Prospects}, Ohio State University Mathematical Research
  Institute Publications, vol.~11, de Gruyter, Berlin, 2005, pp.~257--272.

\bibitem{JL85}
H.~Jacquet and K.~F. Lai.
\newblock {\em A Relative Trace Formula}.
\newblock Compositio Math. 54 (1985), no.~2, 243--310.

\bibitem{JLR93}
H.~Jacquet, K.~F. Lai, and S.~Rallis.
\newblock {\em A Trace Formula for Symmetric Spaces}.
\newblock Duke Math. J. 70 (1993), no.~2, 305--372.

\bibitem{Li23}
H.~Li.
\newblock {\em A Local Trace Formula for $p$-Adic Infinitesimal Symmetric
  Spaces: The Case of Guo--Jacquet}.
\newblock Int. Math. Res. Not. IMRN 2023, no.~24, 21581--21689.

\bibitem{RR96}
C.~Rader and S.~Rallis.
\newblock {\em Spherical Characters on $p$-Adic Symmetric Spaces}.
\newblock Amer. J. Math. 118 (1996), no.~1, 91--178.

\bibitem{Ric82}
R.~W. Richardson.
\newblock {\em Orbits, Invariants, and Representations Associated to
  Involutions of Reductive Groups}.
\newblock Invent. Math. 66 (1982), no.~2, 287--312.

\bibitem{Spa09}
J.~Sparling.
\newblock {\em On the $\theta$-Split Side of the Local Relative Trace
  Formula}.
\newblock Preprint, arXiv:0905.0034, 2009.

\bibitem{Wal95}
J-L. Waldspurger.
\newblock {\em Une Formule des Traces Locale pour les Alg\`{e}bres de Lie
  $p$-adiques}.
\newblock J. Reine Angew. Math. 465 (1995), 41--100.

\end{thebibliography}
\end{document}